\documentclass[11pt, reqno]{amsart}
\usepackage{amsmath,amsopn,amssymb,amsthm,multicol}
\usepackage[cp1251]{inputenc}
\usepackage[english,russian]{babel}
\usepackage[breaklinks=true,colorlinks=true,linkcolor=blue,citecolor=blue,urlcolor=blue]{hyperref}
\usepackage[dvips]{graphicx}

\renewenvironment{proof}[1][Proof]{\textbf{#1.} } {\ \rule{0.5em}{0.5em}}

\newtheorem{theorem}{Theorem}

\newtheorem{lem}{Lemma}
\newtheorem{cor}{Corollary}

\newtheorem{remark}{Remark}

\begin{document}
\selectlanguage{english}

\title[On topological structure of some sets \dots] {On topological structure of some sets related to the normalized Ricci flow on generalized Wallach spaces}
\author{N.\,A.~Abiev}
\address{N.\,A.~Abiev\newline
Taraz State University after M.\,Kh.~Dulaty, \newline
Taraz, Tole~bi str., 60, 080000, KAZAKHSTAN}
\email{abievn@mail.ru}

\begin{abstract} We study topological structures of the sets $(0,1/2)^3 \cap \Omega$ and $(0,1/2)^3 \setminus \Omega$,
where~$\Omega$ is one special algebraic surface defined by a symmetric polynomial in variables  $a_1,a_2,a_3$ of degree~$12$.
These problems arise in studying of general properties of degenerate singular points of
dynamical systems obtained from the normalized Ricci flow  on generalized Wallach spaces.
Our main goal is to prove the connectedness of $(0,1/2)^3 \cap \Omega$
and to determine the number of connected components of $(0,1/2)^3 \setminus \Omega$.

\vspace{2mm} \noindent Key words and phrases:
Riemannian metric, generalized Wallach space, normalized Ricci flow,
dynamical system, degenerate singular point of dynamical system,
real algebraic surface, singular point of real algebraic surface.

\vspace{2mm}

\noindent {\it 2010 Mathematics Subject Classification:} 53C30, 53C44, 37C10, 34C05, 14P05, 14Q10.
\end{abstract}

\maketitle

\section*{Introduction and the main result}

It is known  that determining the connectedness (or the number of connected components)
of real algebraic surfaces is a very hard classical problem in algebraic geometry (see e.~g. \cite{Basu}, \cite{Silhol}).
In this paper we deal with similar problems relating to the normalized Ricci flow on generalized Wallach spaces.
The importance of these problems is due to the need to develop a special apparatus for studying general
properties of degenerate singular points of Ricci flows initiated in \cite{AANS}--\cite{AANS3}.
More concretely, in the above papers, the authors considered some problems concerning the topological structure of the sets
$(0,1/2)^3 \cap \Omega$ and $(0,1/2)^3 \setminus \Omega$,
where
\begin{equation*}\label{surf_Omega}
\Omega =\{(a_1,a_2,a_3)\in\mathbb{R}^3 \, | \,  Q(a_1,a_2,a_3)=0\}
\end{equation*}
is an algebraic surface (see Fig. \ref{singsur} and \ref{singsur_new}) in $\mathbb{R}^3$
defined by a symmetric polynomial $Q(a_1,a_2,a_3)$ in $a_1,a_2,a_3$
of degree $12$:
\begin{eqnarray}\label{singval2}\notag
Q(a_1,a_2,a_3)\,=\,
(2s_1+4s_3-1)(64s_1^5-64s_1^4+8s_1^3+12s_1^2-6s_1+1\\\notag
+240s_3s_1^2-240s_3s_1-1536s_3^2s_1-4096s_3^3+60s_3+768s_3^2)\\
-8s_1(2s_1+4s_3-1)(2s_1-32s_3-1)(10s_1+32s_3-5)s_2\\\notag
-16s_1^2(13-52s_1+640s_3s_1+1024s_3^2-320s_3+52s_1^2)s_2^2\\\notag
+64(2s_1-1)(2s_1-32s_3-1)s_2^3+2048s_1(2s_1-1)s_2^4,
\end{eqnarray}
$$
s_1 = a_1+a_2+a_3, \quad s_2 = a_1a_2+a_1a_3+a_2a_3, \quad s_3 = a_1a_2a_3.
$$

\begin{center}
\begin{figure}[t]
\centering\scalebox{1}[1]{
\includegraphics[angle=-90,totalheight=2.6in]{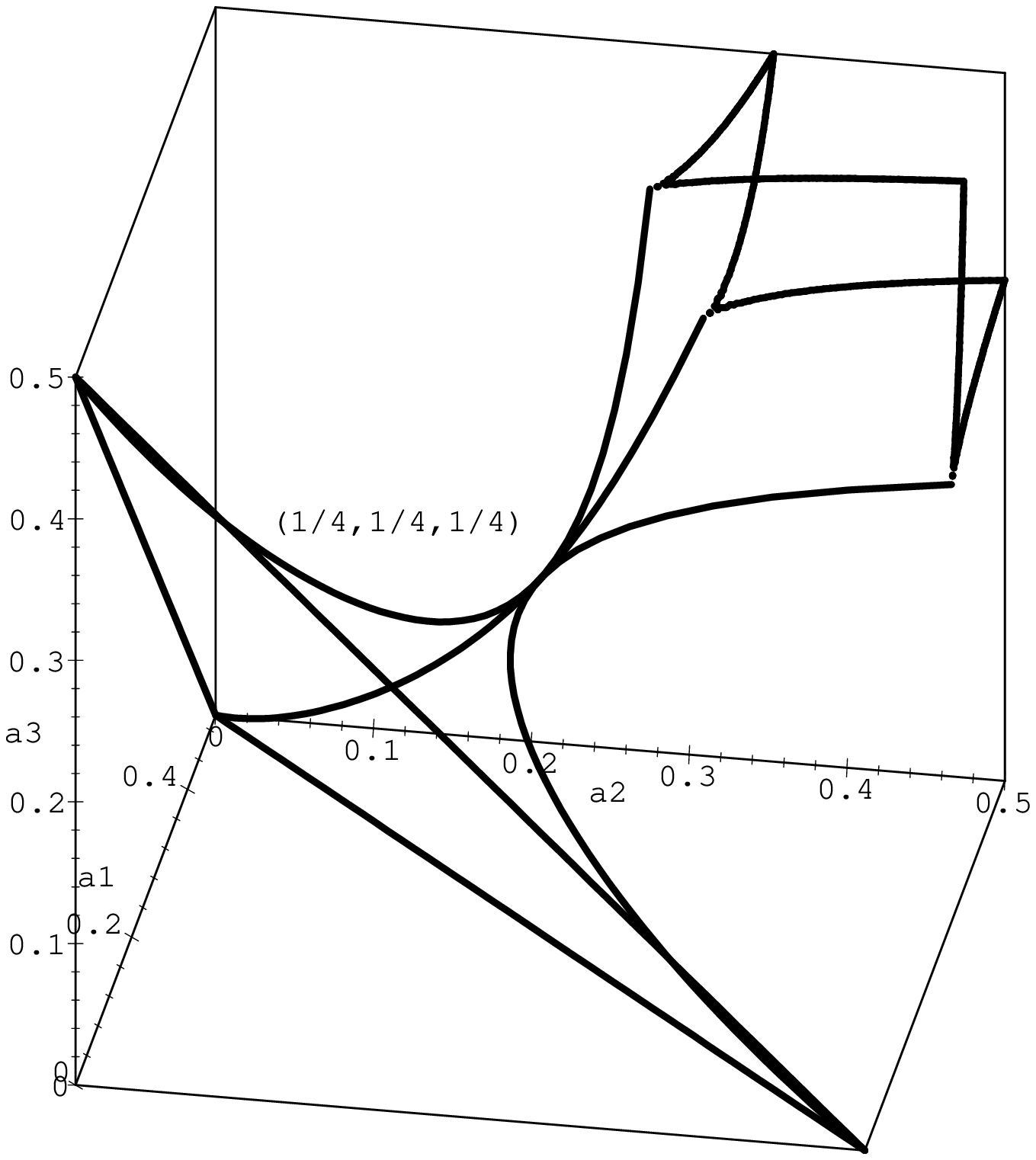}
\includegraphics[angle=-90,totalheight=2.6in]{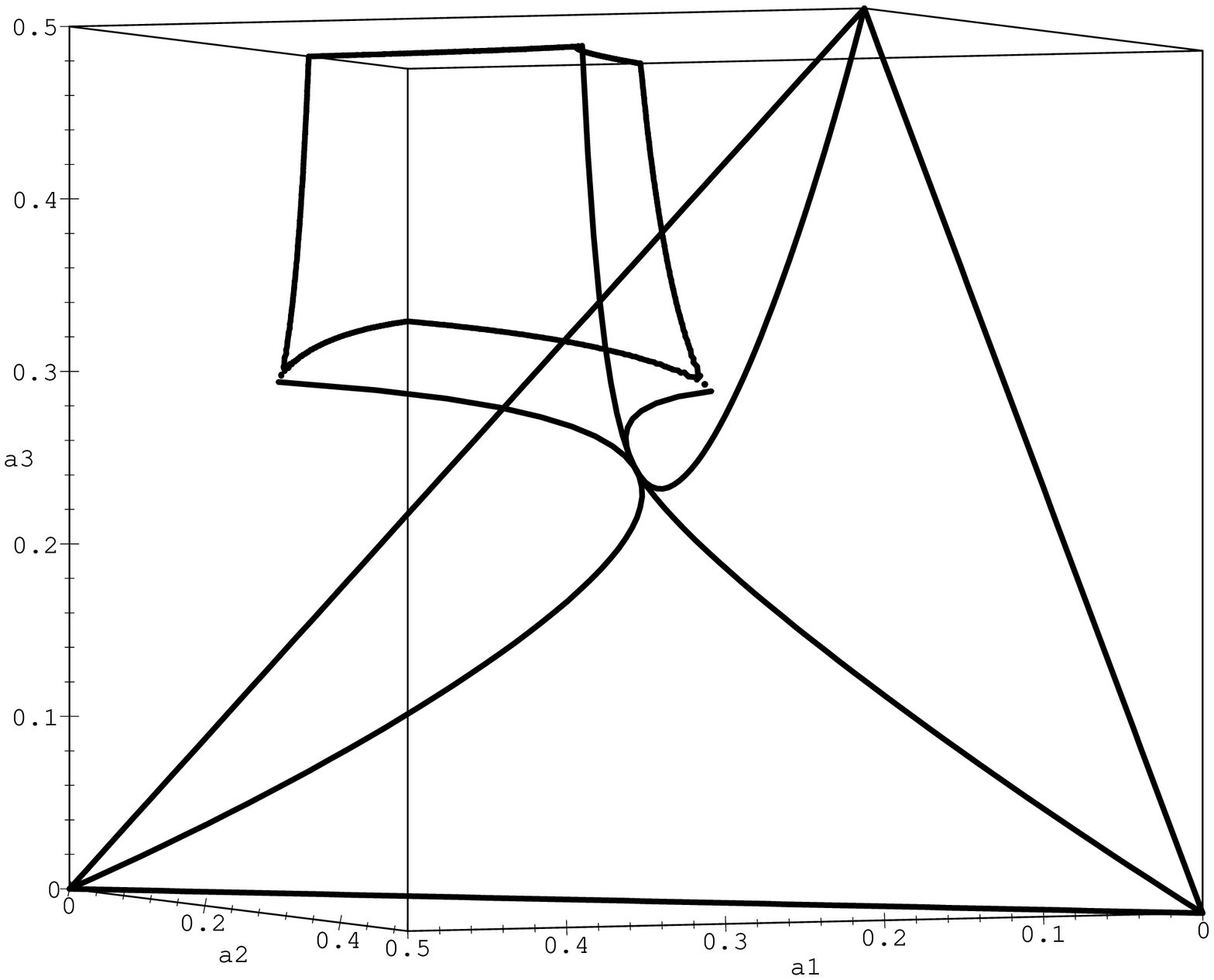}}
\caption{Singular points of the surface $(0,1/2)^3 \cap \Omega$}
\label{singsur}
\end{figure}
\end{center}

The surface $\Omega$ naturally arises in studying of general properties of degenerate singular points of the following
dynamical system (see \cite{AANS}--\cite{AANS3}):
\begin{equation}
\label{three_equat}
\dfrac {dx_1}{dt} = f(x_1,x_2,x_3), \quad
\dfrac {dx_2}{dt}=g(x_1,x_2,x_3), \quad
\dfrac {dx_3}{dt}=h(x_1,x_2,x_3),
\end {equation}
where  $x_i=x_i(t)>0$ $(i=1,2,3)$,
\begin{eqnarray*}
f(x_1,x_2,x_3)&=&-1-a_1x_1 \left( \dfrac {x_1}{x_2x_3}-  \dfrac {x_2}{x_1x_3}- \dfrac {x_3}{x_1x_2} \right)+x_1B,\\
g(x_1,x_2,x_3)&=&-1-a_2x_2 \left( \dfrac {x_2}{x_1x_3}- \dfrac {x_3}{x_1x_2} -  \dfrac {x_1}{x_2x_3} \right)+x_2B,\\
h(x_1,x_2,x_3)&=&-1-a_3x_3 \left( \dfrac {x_3}{x_1x_2}-  \dfrac {x_1}{x_2x_3}- \dfrac {x_2}{x_1x_3} \right)+x_3B,
\end{eqnarray*}
$$
B:=\left( \dfrac {1}{a_1x_1}+\dfrac {1}{a_2x_2}+\dfrac {1}{a_3x_3}- \left( \dfrac {x_1}{x_2x_3}+
\dfrac {x_2}{x_1x_3}+ \dfrac {x_3}{x_1x_2} \right) \right)
\left( \frac{1}{a_1} +\frac{1}{a_2}+ \frac{1}{a_3} \right)^{-1}.
$$
$$
a_i \in (0,1/2] \quad (i=1,2,3).
$$

\begin{center}
\begin{figure}[t]
\centering\scalebox{1}[1]{
\includegraphics[angle=0,totalheight=3.5in]{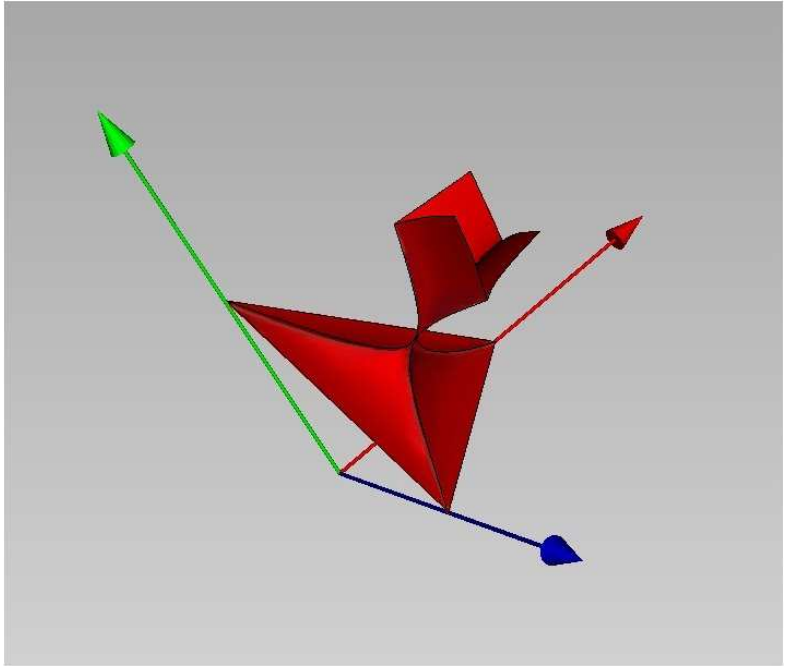}
}
\caption{The surface $(0,1/2)^3 \cap \Omega$}
\label{singsur_new}
\end{figure}
\end{center}

It should be noted that the system (\ref{three_equat}) can be obtained
from the normalized Ricci flow equation
$$
\dfrac {\partial}{\partial t} \bold{g}(t) = -2 \operatorname{Ric}_{\bold{g}}+ 2{\bold{g}(t)}\frac{S_{\bold{g}}}{n},
$$
where $\bold{g}(t)$ means a $1$-parameter family of Riemannian metrics,
$\operatorname{Ric}_{\bold{g}}$ is the Ricci tensor and $S_{\bold{g}}$ is the scalar curvature of the Riemannian metric ${\bold{g}}$,
considered on one special class of compact homogeneous spaces
called three-locally-symmetric or generalized Wallach spaces, see \cite{Lomshakov1}, \cite{Nikonorov1}.
In the recent papers \cite{CKL} and \cite{Nikonorov4}, the complete classification of these spaces was obtained.

A more detailed information concerning  geometric aspects of this problem and the Ricci flows
could be found in  \cite{Lomshakov1},\cite{Nikonorov2}, \cite{ChowKnopf} and \cite{Topping}.
\medskip

In  \cite{AANS}, the authors noted that {\it the set  $(0,1/2)^3 \cap \Omega$ is connected, and the set $(0,1/2)^3\setminus \Omega$
consists of three connected components  $O_1$, $O_2$ and $O_3$ (see Fig. \ref{singsur})
containing the points
$(1/6,1/6,1/6)$, $(7/15,7/15,7/15)$ and $(1/6, 1/4, 1/3)$ respectively.}

\smallskip
The present work is devoted to detailed proof of this observation. The main result is the following

\begin{theorem}\label{main_thm}
The following assertions hold with respect to the standard  topology of $\mathbb{R}^3$:
\begin{enumerate}
\item
The set $(0,1/2)^3 \cap \Omega$ is connected;
\item
The set  $(0,1/2)^3\setminus \Omega$ consists of three connected components.
\end{enumerate}
\end{theorem}

We note also the following

\begin{cor}\label{main_cor} The assertions of Theorem \ref{main_thm} are preserved  if
$(0,1/2)^3$ is replaced by  $(0,1/2]^3$.
\end{cor}

\begin{remark}\label{symm_Omega}
The symmetry of $Q$ with respect to $a_1,a_2,a_3$
implies the invariance of $\Omega$ under the permutation
$a_1\rightarrow a_2 \rightarrow a_3\rightarrow a_1$.
\end{remark}

\begin{remark}\label{Idea}
Proof of Theorem \ref{main_thm} is based on the idea of Remark 8 in \cite{AANS2}:
One should consider a segment $I$ with one endpoint at $(0,0,0)$ and with  the second endpoint at
an arbitrary point of any facet of the cube  $(0,1/2)^3$ containing  $(1/2,1/2,1/2)$.
According to Remark \ref{symm_Omega},
we can assume without loss of generality that $I$ is defined by the following parametric equations
\begin{equation} \label{parametrization}
a_1:=at, \quad a_2:=bt, \quad a_3:=t/2,
\end{equation}
where $t\in [0,1]$, \, $a,b\in (0,1/2)$.
Substituting  (\ref{parametrization}) into (\ref{singval2}) we obtain some polynomial $p(t):=Q(at,bt,t/2)$ in $t$ of degree $12$.
Thus the problems under consideration could be reduced to the problem of determining the possible number
of roots of $p(t)$ in $[0,1]$ when $(a,b)\in (0,1/2)^2$.
\end{remark}

\section{Proof of the main result}

Using Maple we have the following explicit expression for $p(t)$:
\begin{eqnarray} \notag \label{polynom12}
p(t)=-256b^2a^2(2a+1)^2(2b+1)^2(b+a)^2t^{12}+32(16b^3a^3\\ \notag
+4b^3a^2+2b^3a+2b^3+8b^2a^2+b^2a+4b^2a^3+2ba^3+ba^2+2a^3)\\ \notag
(2a+1)(2b+1)(2b+1+2a)(b+a)t^{10}\\ \notag
-32(2a+1)(2b+1)(b+a)(16b^3a^3+4b^3a^2+2b^3a+2b^3+8b^2a^2\\ \notag
+b^2a+4b^2a^3+2ba^3+ba^2+2a^3)t^9\\ \notag
-(72b^2a^2+104ba^3+208b^3a^2+104b^3a+208b^2a^3+52b^4+176a^4b+208b^4a^2\\ \notag
+176b^4a+52ba^2+52b^2a+208a^4b^2+52a^4+352b^3a^3+13b^2\\ \notag
+13a^2+44a^3+44b^3+22ba)(2b+1+2a)^2t^8\\ \notag
+2(2b+1+2a)(72b^2a^2+104ba^3+208b^3a^2+104b^3a+208b^2a^3+52b^4\\ \notag
+176a^4b+208b^4a^2+176b^4a+52ba^2+52b^2a+208a^4b^2+52a^4\\
+352b^3a^3+13b^2+13a^2+44a^3+44b^3+22ba)t^7\\ \notag
+(600b^2a^2+392ba^3+784b^3a^2+392b^3a+784b^2a^3+108b^4+14b+14a\\ \notag
+128a^6+448ba^5+224a^5+528a^4b+432b^4a^2+528b^4a+196ba^2\\ \notag
+196b^2a+432a^4b^2+108a^4+288b^3a^3+224b^5+448b^5a+128b^6\\ \notag
+2+27b^2+27a^2+36a^3+36b^3+66ba)t^6\\ \notag
-6(8b^3+4b^2a+2b^2+8ba+b+4ba^2+2a^2+8a^3+1+a)(2b+1+2a)^2t^5\\ \notag
+(2b+1+2a)(40b^3+24ba+5+40a^3)t^4\\ \notag
+(22b+22a+88ba^2+88b^2a+2+44b^2+44a^2+16a^3+16b^3+80ba)t^3\\ \notag
-6(2b+1+2a)^2t^2+(8a+8b+4)t-1.
\end{eqnarray}

Consider the following set
$$
K:=\left\{(a,b)\in \mathbb{R}^2~|~a,b\in (0,1/2)\right\}.
$$
\begin{center}
\begin{figure}[t]
\centering
\scalebox{1}[1]{\includegraphics[angle=-90,totalheight=3in]
{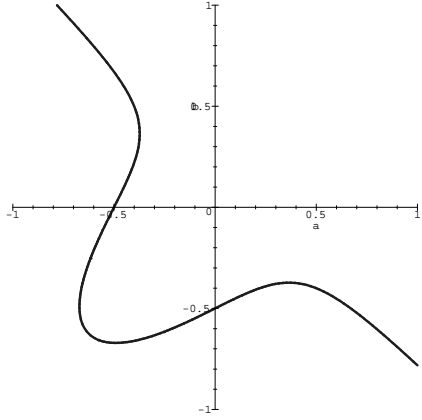}}
\caption{The curve $\gamma$}
\label{F}
\end{figure}
\end{center}

\begin{lem} \label{Discrim_of_p(t)}
If $(a,b)\in K$  then the discriminant $D$ of the polynomial $p(t)$ equals to zero if and only if $a=b$.
\end{lem}

\begin{proof}
Easy calculations show that  $D$  is non-negative,
moreover, $D$ has the same zeroes as the following polynomial:
\begin{equation}\label{resultant}
(2b-1)^{12}(2a-1)^{12}(a-b)^{12}\big(F(a,b)\big)^2,
\end{equation}
where
\begin{equation*}\label{F(a,b)}
F(a,b):=40a^3-24a^2b-24ab^2+40b^3-12a^2+12ba-12b^2-6a-6b+5.
\end{equation*}

Denote by $\gamma$ the curve  determined by $F(a,b)=0$ (see Fig. \ref{F}).
We will prove that $\gamma$ has no common point with the square $K$.

Changing the variables by the formula
$$
x-y=a\sqrt{2}, \quad x+y=b\sqrt{2}\,,
$$
we get a new equation for $\gamma$, from that we can express $y$ explicitely:
\begin{equation}\label{new_equation}
\widetilde{F}(x,y):=36\left(8x-\sqrt{2}\right)y^2+\left(8x+5\sqrt{2}\right)\left(2x-\sqrt{2}\right)^2=0.
\end{equation}

Note that the point $(x',y')=\left(\sqrt 2/2,0\right)$ belongs to $\gamma$,
moreover, this is an unique singular point of $\gamma$.
Since
$$
\widetilde{F}_{xx}\widetilde{F}_{yy}-\widetilde{F}_{xy}^2=3888>0
$$
at $(x',y')$, then $(x',y')$ is isolated according to the well-known result in differential geometry of planar curves.
It is clear that the point $(a,b)=(1/2,1/2) \notin K$  corresponds to $(x',y')$ in the initial variables.

It is obvious that every regular point of $\gamma$ satisfies the condition $x<x_0:=\sqrt{2}/8$.
Hence only we need is to show that
$\gamma$ can not intersect  the part of $K$, described by the conditions
$x \in\left (0,x_0\right)$,\, $-x<y<x$.
In fact, it suffices to prove the inequality
$x<\varphi (x)$,
where
$$
\varphi (x):= \frac{\sqrt2-2x}{6}\sqrt{\frac{8x+5\sqrt{2}}{\sqrt{2}-8x}}
$$
is a function determining a part of the curve $\gamma$ in (\ref{new_equation}).
Note that $\lim\limits_{x\rightarrow x_0-0}\varphi(x)=+\infty$.

It is easy to show that the inequality  $x<\varphi (x)$ is equivalent to the inequality
$$
\psi(x):=320x^3-48\sqrt 2\, x^2-24x+10\sqrt 2>0,
$$
which holds for all $x \in\left (0,x_0\right)$, since  $\psi(x)$ is positive at $x=x_0$ and decreases:
$$
  \psi\left(x_0\right)=27\sqrt2/4>0,  \qquad \psi'(x)=960x^2-96\sqrt 2x-24<0.
$$

Therefore, $F(a,b)\ne 0$  for $(a,b)\in K$.
Hence there is a unique possibility $a=b$  in order to $D=0$ in $K$ by (\ref{resultant}).
\end{proof}
\bigskip

\begin{lem} \label{multiple_isnot_extremum} Let $(a,b)\in K$.
Then a point of local extremum of $p(t)$ can not be a multiple root of $p(t)$.
\end{lem}

\begin{proof}
Multiple roots of $p(t)$ are possible only for $a=b$ by Lemma \ref{Discrim_of_p(t)}.
Therefore, we may assume that $b=a$.
Then   (\ref{polynom12}) takes the following form
$$
p(t)=-(t+1)\,p_2(t)\,p_3^3(t),
$$
\begin{eqnarray*}
p_2(t) &:=& (2+4a)t^2-2(1+2a)t+1, \\
p_3(t) &:=& 8a^2(2a+1)t^3-(1+4a)t+1.
\end{eqnarray*}

Denote by  $D_2$ and $D_3$ the discriminants of  $p_2(t)$ and $p_3(t)$ respectively:
\begin{eqnarray*}
D_2&:=&4(2a+1)(2a-1),\\
D_3&:=&-32(2a+1)(2a-1)(22a^2+14a+1)a^2.
\end{eqnarray*}

Since  $D_2<0,\, D_3>0$ for $a\in (0,1/2)$, then it is clear that
the polynomial  $p(t)$ has exactly three distinct real roots (each of multiplicity $3$) for every such $a$.
It follows from this fact that there is no points of local extrema of $p(t)$ among the roots of $p(t)$.
\end{proof}

\smallskip
Further we need  the curve $\Gamma$ (see Fig. \ref{pictur1}),
which can be obtained as a result of the intersection $\Omega$ with the plane $a_3=1/2$ for $0<a_1,a_2\le 1/2$.
Recall some properties of $\Gamma$ (see details in  \cite{AANS2}):
$\Gamma$ determined by the equality $G(a_1,a_2)=0$, where
\begin{multline}\label{G_ab}
G(a_1,a_2):=4(a_1+a_2)(4a_1a_2-1)(4a_1a_2-a_1-a_2+1)(4a_1a_2+a_1+a_2+1)\\
+(16a_1^2a_2^2+1)(13a_1^2+22a_1a_2+13a_2^2)-4(a_1^2+a_2^2)(11a_1^2+18a_1a_2+11a_2^2),
\end{multline}
$\Gamma$ is homeomorphic to the segment $[0,1]$ with the endpoints $(\sqrt 2/4, 1/2)$, $(1/2, \sqrt 2/4)$ and
with the unique singular point (a cusp) at
 $(a_1,a_2)=(\tilde a, \tilde a)$, where
 $\tilde a:=(\sqrt5 -1)/4\approx 0.3090169942$.

It is easy to check that $\Gamma$ separates  $K$ into disjoint connected components  $K_1$ and $K_2$
containing the points
$$
(a',b'):=(3/10,3/10)\quad \mbox{ and }\quad (a'',b''):=(31/100,31/100)
$$ respectively.

 \begin{center}
\begin{figure}[t]
\centering
\scalebox{1}[1]{\includegraphics[angle=-90,totalheight=3in]
{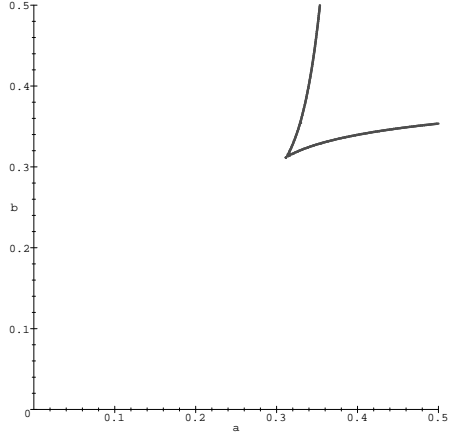}}
\caption{The intersection $\Omega$ with the plane $a_3=1/2$ for $0<a_1,a_2\le 1/2$}
\label{pictur1}
\end{figure}
\end{center}
\bigskip

\begin{lem} \label{number_ofRoots_of_p(t)}
In the segment $[0,1]$, the polynomial $p(t)$  has
\begin{enumerate}
\item
one root, if  $(a,b)\in K_1$;
\item
two distinct roots, if  $(a,b)\in K_2 \cup \Gamma$.
\end{enumerate}
\end{lem}

\begin{proof}
Let $t^\ast\in [0,1]$ be a root of  $p(t)$ given by (\ref{polynom12}).
We say that  $t^\ast$ is a robust root of $p(t)$  in~$[0,1]$, if
small  perturbations of the parameters
$a$ and $b$ imply a small perturbation of $t^\ast$ keeping it
in $(t^\ast-\varepsilon, t^\ast+\varepsilon)\subset[0,1]$ for some small $\varepsilon>0$
(see e.~g.~\cite{Bruce} for more details on singularities of curves and some related problems).

\smallskip
Now, assume that  $t^\ast$ is a non-robust root of $p(t)$.
Then there exist exactly two possibility (recall that $t^\ast\in [0,1]$):
\smallskip

{\it Case 1.}  $t^\ast=0$ or $t^\ast=1$;
\smallskip

{\it Case 2.}
$t^\ast$ belongs to the interval $(0,1)$ and provides  $p(t)$ a local extremum.

\medskip
Now, we consider these cases separately.
\smallskip

{\it Case 2.}  Assume that   $t^\ast$ is a point of local extremum of  $p(t)$.
Then  $t^\ast$ is a multiple root of $p(t)$.
This contradicts to Lemma  \ref{multiple_isnot_extremum}, hence, the case~2 is impossible.
\smallskip

{\it Case  1.}
Since $p(0)=-1$ then there exists no pair $(a,b)$ such that  $t=0$ is a root of  $p(t)$.

Suppose that $t=1$ is a root of  $p(t)$.
Since
$$
p(1)=-4(a+b)^2G(a,b),
$$
where $G$ is given by (\ref{G_ab}),
then the equality $p(1)=0$ is possible if and only if $G(a,b)=0$.

Recall that the curve $\Gamma$ is determined by $G(a,b)=0$.
Since $p(t)$ has only robust roots for every pair $(a,b)\in K_1\cup K_2$ by our construction,
then the number of roots of $p(t)$ in $[0,1]$ is constant both in $K_1$ and in $K_2$.
Hence, it is sufficient to calculate the number of such roots
only for the representative points $(a',b')\in K_1$ and $(a'',b'')\in K_2$.
\smallskip

$(1)$ Suppose that  $(a,b)=(a',b')\in K_1$.
Then (\ref{polynom12}) takes the following form
$$
p(t)= -\frac{1}{9765625}(t+1)(16t^2-16t+5)(144t^3-275t+125)^3.
$$

Taking into account Lemma \ref{multiple_isnot_extremum}, we conclude that
$p(t)$ has three distinct real roots of multiplicity  $3$  besides the root  $t=-1$.
Since we does not need exact values of these roots then
their approximated values are:
$$
-1.569348118, \quad 0.5345099430, \quad 1.034838175.
$$

\smallskip

$(2)$ Now, suppose that $(a,b)=(a'', b'')\in K_2$.
Then  in (\ref{polynom12}) we obtain
$$
p(t) = -\frac{1}{6103515625000000}(t+1)(81t^2-81t+25)(77841t^3-140000t+62500)^3,
$$
with the following real roots (of multiplicity $3$):
$$
 -1.524828329\dots,\quad 0.5285082631\dots,\quad 0.9963200660\dots.
$$

It is easy to see that  for $(a,b) \in \Gamma$ the polynomial  (\ref{polynom12}) has two
roots in $[0,1]$, one of which is $1$ by the definition of $\Gamma$.

Hence, in the segment $[0,1]$, the polynomial (\ref{polynom12}) has one root for $(a,b)\in K_1$ and two roots for $(a,b)\in K_2 \cup \Gamma$.
\end{proof}
\medskip

\bigskip

{\bf Proof of Theorem  \ref{main_thm}} is based on Lemma  \ref{number_ofRoots_of_p(t)} and Remark \ref{Idea}.
Let  $(a,b)\in K$. Then the number of intersection points of $\Omega$
with the segment  $I$ equals to $1$ or $2$  depending on the number of roots of the polynomial $p(t)$ (see (\ref{polynom12}))
containing in $[0,1]$.

\smallskip

{\bf (1)} {\it Connectedness of the set $(0,1/2)^3 \cap \Omega$}.
Let $t_1,t_2$ be roots of $p(t)$ such that $0<t_1 < t_2\le 1$.
Then, obviously,  $t_1$ and $t_2$ correspond to the ``lower'' and ``upper'' (see Fig.~\ref{singsur_new}) parts of the surface $\Omega \cap (0,1/2)^3$ respectively.
These parts of $\Omega$ have a unique common point $(a_1,a_2,a_3)=(1/4,1/4,1/4)$
(an {\it elliptic umbilic} of $\Omega$ according to  \cite{AANS}).
\smallskip

{\bf (2)} {\it The number of the connected components of the set $(0,1/2)^3 \setminus \Omega$}.
Since the maximal number of roots of  $p(t)$  in $[0,1]$ is equal to $2$ and $\Omega \cap (0,1/2)^3$ is the union of two surfaces with one common point,
then the number of connected components of  $(0,1/2)^3 \setminus \Omega$ equals to~$3$. Theorem  \ref{main_thm} is proved.

\bigskip
In order to prove Corollary   \ref{main_cor} we need the following

\begin{lem}\label{number_ofRoots_of_p(t)_b=1/2}
Let $b=1/2$. Then in the segment $[0,1]$, the polynomial $p(t)$  has
\begin{enumerate}
\item
one root for  $a\in \left(0,\sqrt 2/4\right)$;
\item
two roots for $a\in \left[\sqrt 2/4,1/2\right)$;
\item
one root (of multiplicity  $8$) for  $a=1/2$.
\end{enumerate}
\end{lem}

\begin{proof} $(1),(2)$ At  $b=1/2$, $a\in (0,1/2)$ we have
$$
p(t)=-(2ta+1)\,p_2(t)\,p_3^3(t)
$$
in (\ref{polynom12}), where
\begin{eqnarray*}
p_2(t) &:=& 4a(2a+1)t^2-2(1+2a)t+1,\\
p_3(t) &:=&2 (1+2a)t^3-2(a+1)t+1.
\end{eqnarray*}

For the discriminants  $D_2$ and $D_3$ of the polynomials $p_2(t)$ and $p_3(t)$ we have
\begin{eqnarray*}
D_2&:=&-4(2a-1)(2a+1)>0,\\
D_3&:=&4(2a+1)(2a-1)(8a^2+28a+11)<0.
\end{eqnarray*}

Since
the cubic polynomial $p_3(t)$ achieves a positive local maximum at the point $t=-\frac{(6a+3)(a+1)}{6a+3}<0$,
then its unique real root must be a negative number.
Therefore, required roots of  $p(t)$ can be provided only by $p_2(t)$,
moreover, first of them belongs to  $[0,1]$ for all  $a\in(0,1/2)$;
the second of them   --- only for $a\in \left[\sqrt 2/4,1/2\right)$.

\smallskip

$(3)$ The case $b=a=1/2$ leads (\ref{polynom12}) to the polynomial
$$
p(t)=-(t+1)^4(2t-1)^8
$$
with the unique root $t=1/2$ of multiplicity  $8$ on $[0,1]$.
It should be noted that  we get an elliptic umbilic  $(a_1,a_2,a_3)=(1/4,1/4,1/4)$
of the surface $\Omega$ in this case.
\end{proof}

\bigskip

{\bf Proof of Corollary  \ref{main_cor}}.  According to Theorem  \ref{main_thm}
it is sufficient to consider the case when $a=1/2$ or $b=1/2$.
Taking into account Remark  \ref{symm_Omega}, assume without loss of generality that  $b=1/2$.
Then the proof of Corollary  \ref{main_cor}  follows from Lemma \ref{number_ofRoots_of_p(t)_b=1/2}
and Remark \ref{Idea}.

\begin{remark}
When this paper has been written the author was informed about the recent preprint \cite{Batkhin},
where a more detailed description of the surface $\Omega$ was obtained
without the restriction $(a_1,a_2,a_3)\in (0,1/2)^3$.
\end{remark}

\bigskip

The author is indebted to Prof. Yu.\,G.\,Nikonorov and to Prof. A.\,Arvanitoyeorgos for helpful discussions concerning this paper.

\vspace{10mm}

\bibliographystyle{amsunsrt}

\vspace{5mm}
\end{document}